\documentclass[12pt,epsfig,amsfonts]{amsart} 
\usepackage{amsmath,amsthm,amssymb,amscd,epsfig,lmodern}
\usepackage{graphicx}
\usepackage{mathrsfs}
\usepackage[T1]{fontenc}
\usepackage[utf8]{inputenc}
\usepackage{draftwatermark}

\newtheorem*{theorema}{Theorem A}
\newtheorem*{theoremb}{Theorem B}
\newtheorem*{theoremc}{Theorem C}
\newtheorem*{theoremd}{Theorem D}
\newtheorem*{cor1}{Corollary 1}
\newtheorem*{cor2}{Corollary 2}

\newtheorem{lemma}{Lemma}[section]

\newtheorem{remark}[lemma]{Remark}

\newtheorem{prop}[lemma]{Proposition}

\begin{document}

\author{Hiroki Takahasi}
\address{Department of Mathematics,
Keio University, Yokohama,
223-8522, JAPAN} 
\email{hiroki@math.keio.ac.jp}
\urladdr{\texttt{http://www.math.keio.ac.jp/~hiroki/}}

\thanks{This research was partially supported by 
the Grant-in-Aid for Young Scientists (A) of the JSPS, 15H05435,
 the Grant-in-Aid for Scientific Research (B) of the JSPS, 16KT0021,
 and the JSPS Core-to-Core Program ``Foundation
of a Global Research Cooperative Center in Mathematics focused on Number Theory and Geometry''.
I thank Mike Todd for fruitful discussions.}

\title[Equilibrium states at freezing phase transition] 
{Equilibrium states at freezing phase transition\\
 in unimodal maps with flat critical point}
\date{\today}

\begin{abstract}
An $S$-unimodal map $f$ with flat critical point satisfying the Misiurewicz condition
displays a freezing phase transition in positive spectrum.
We analyze statistical properties of the equilibrium state $\mu_t$ for the potential $-t\log|Df|$,
as well as how the phase transition slows down the rate of decay of correlations.
 We show that $\mu_t$ has exponential decay of correlations for all inverse temperature $t$ contained in 
 the positive entropy phase $(t^-,t^+)$. If the critical point is not too flat, 
 then the freezing point $t^+$ is equal to $1$, and the absolutely continuous invariant probability measure (acip for short) is the unique equilibrium state at the transition. 
We exhibit a case in which the acip has sub-exponential decay of correlations
 and $\mu_t$ converges weakly to the acip as $t\nearrow t^+$.
\end{abstract}
\maketitle

\section{Introduction}

It is now classical that the thermodynamic formalism for topologically mixing uniformly hyperbolic systems
yields a rather satisfactory description of the dynamics. 
The existence of a finite Markov partition allows one to ``code'' the orbits of the system by following their histories over the partition
of the phase space.
This defines a one-dimensional spin system with exponentially decaying interactions, 
and the construction of natural invariant measures can be done on a symbolic level.
According to the pioneering works of Sina{\u\i}, Ruelle and Bowen \cite{Bow75,Rue78,Sin72},
 there exists a unique equilibrium state at all inverse temperature,
the speed of convergence to equilibrium is exponential, 
and there is no phase transition. 

The purpose of equilibrium statistical physics is to understand phase transitions.
Therefore, advancing the thermodynamic formalism beyond the realm of uniform hyperbolicity is important.
For non-hyperbolic systems, nice thermodynamic properties are limited to the inside of 
 positive entropy phases, and at their boundaries occur phase transitions.
One archetypal model is  given by 
interval maps with indifferent fixed point such as the Manneville-Pomeau map \cite{PomMan80},
also known as a geometric realization of Hofbauer's model \cite{Hof77}.
For this class of maps, Prellberg and Slawny \cite{PreSla92} established exponential decay of correlations
of equilibrium states 
in positive entropy phase. At the transition there exist two ergodic equilibrium states, and
the one that is absolutely continuous with respect to the Lebesgue measure has sub-exponential decay of correlations,
from the results of Sarig \cite{Sar02}, Gou\"ezel \cite{Gou04}, Hu \cite{Hu04}.
The aim of this paper is to show the occurrence of an analogous phenomenon for certain unimodal maps.
The map we treat has a non-recurrent and flat critical point, as defined below.

Let $X=[0,1]$ and
 $f\colon X\to X$ be
{\it a unimodal map}, i.e., a $C^1$ map whose critical set $\{x\in X\colon Df(x)=0\}$ consists of a single point $c\in(0,1)$ that is the extremum. 
 {\it An $S$-unimodal map} $f$ is a unimodal map of class $C^3$
on $X\setminus\{c\}$
with negative Schwarzian derivative.
For an $S$-unimodal map $f$
 let $\mathcal M(f)$ denote the set of $f$-invariant Borel probability measures.
For each $\mu\in\mathcal M(f)$ the Kolmogorov-Sina{\u\i} entropy of $(f,\mu)$ is denoted by $h(\mu)$.
      Define {\it a geometric pressure function} $t\in\mathbb R\mapsto \mathscr P(t)$ by
$$\mathscr P(t)=\sup\left\{h(\mu)-t\int\log|Df|d\mu\colon\mu\in\mathcal M(f)\right\}.$$
Measures which attain this supremum are called {\it equilibrium states} for the potential $-t\log|Df|$.
As in the case of the uniformly hyperbolic systems,  
a good deal of information is encoded in the pressure function and the equilibrium states. 

Set
$$\chi_{\rm inf}=\inf\left\{\int\log|Df|d\mu\colon\mu\in\mathcal M(f)\right\}\ \ \text{and}\ \ 
\chi_{\rm sup}=\sup\left\{\int\log|Df|d\mu\colon\mu\in\mathcal M(f)\right\},$$
and define {\it a freezing point} and {\it a condensation point} respectively by
\begin{equation*}
t^+=\sup\{t\in\mathbb R\colon \mathscr P(t)+t\chi_{\rm inf}>0\}\ \ \text{and}\ \ 
t^-=\inf\{t\in\mathbb R\colon \mathscr P(t)+t\chi_{\rm sup}>0\}.\end{equation*}
From the result of Bruin and Keller \cite{BruKel98},
 $\chi_{\rm inf}\geq0$ holds\footnote{The proof does not use the non-flatness of the critical point.}
 provided all periodic points of $f$ are hyperbolic repelling, and
  in particular $-\infty\leq t^-<0<t^+\leq+\infty$.
By {\it the positive entropy phase} we mean the interval $(t^-,t^+)$.
Inside the positive entropy phase,
the equilibrium state has positive entropy.
At the boundary the system is ``frozen'', and outside of it any equilibrium state has zero entropy
and the pressure function is linear.

     We say $f$ satisfies {\it the Misiurewicz condition} \cite{Mis81} 
 if the following holds:
 \begin{itemize}
 \item[(M1)] $c\notin\omega(c)$, where $\omega(c)$ denotes the omega-limit set of the critical point $c$;

 \item[(M2)] 
 if $x\in X$ and $n\geq1$ are such that
 $f^n(x)=x$, then $|Df^n(x)|>1$. 
 \end{itemize}
The critical point $c$ is {\it flat} if there exists
a $C^3$ function $\ell$ on $X\setminus\{c\}$ such that:

\begin{itemize}

\item[(F1)] 
$ \ell(x)\to+\infty$ and $|D\ell(x)|\to+\infty$ as $x\to c$.
Here, $x\to c$ indicates both $x\to c+0$ and $x\to c-0$;

\item[(F2)] there exist
$C^1$ diffeomorphisms $\xi$, $\eta$ of $\mathbb R$ such that
$\xi(c)=0=\eta(f(c))$ and $|\eta\circ f(x)|=|\xi(x)|^{\ell(x)}$ for $x$ near $c$.
\end{itemize}

In the study of one-dimensional dynamical systems (iterated maps of an interval or the circle to itself), 
all critical points are often assumed to be {\it non-flat}
(See de Melo and van Strien \cite{dMevSt93} for the definition).
Although this assumption is useful, 
maps with flat critical point are neither pathological nor unimportant.
For instance, Hall \cite{Hal81} essentially used flat critical points to construct a $C^\infty$
Denjoy counterexample.
 For an $S$-unimodal map with flat critical point satisfying
the Misiurewicz condition, Benedicks and Misiurewicz \cite{BenMis89} constructed
a $\sigma$-finite invariant Borel measure that is absolutely continuous with respect to the Lebesgue measure ({\it acim} for short). 
Zweim\"uller \cite{Zwe04} proved various distributional limit theorems.
For a parametrized family of $S$-unimodal maps with flat critical point, Thunberg \cite{Thu99}
proved the existence of a positive measure set of parameters for which the corresponding maps
exhibit an exponential growth of derivatives along the orbit of the critical point.
The thermodynamic formalism for maps with non-flat critical point have now been extended to 
``cusp maps'', which in particular includes maps with flat critical point, see Dobbs 
\cite{Dob14}, Iommi and Todd \cite{IomTod10}.

We study 
an $S$-unimodal map $f$ with flat critical point satisfying the Misiurewicz condition.
For a pair of measurable functions $\phi,\psi\colon X\to\mathbb R$ their {\it correlation} 
with respect to  $\mu\in\mathcal M(f)$ is defined by 
$${\rm Cor}(\phi\circ f^n,\psi)=\int (\phi\circ f^n)\psi d\mu-\int \phi d\mu\int\psi d\mu.$$
We say $\mu$ has {\it exponential decay of correlations}
if for every $\phi\in L^\infty(\mu)$ and every H\"older continuous $\psi$ there exist $A>0$ and $\theta\in(0,1)$ such that
 ${\rm Cor}(\phi\circ f^n,\psi)\leq A\theta^n$ for every $n\geq0$. We say $\mu$ has {\it sub-exponential decay of correlations} 
if there exist $\phi$, $\psi$ for which there are no such $A$ and $\theta$.

 

 \begin{theorema}\label{theorema}
 Let $f\colon X\to X$ be a topologically transitive $S$-unimodal map with flat critical point
satisfying the Misiurewicz condition. Then the following holds:
\begin{itemize}

\item[1.] for every $t\in(t^-,t^+)$ there exists a unique equilibrium state for the potential $-t\log|Df|$
(denoted by $\mu_t$);

\item[2.] if $t\in(t^-,t^+)$ then the following holds for the measure $\mu_t$:
 for every $\eta\in(0,1]$ there exists $r\in(0,1)$ such that for every $\phi\in L^\infty$ and every
 H\"older continuous $\psi\colon X\to\mathbb R$
with H\"older exponent $\eta$, there exists a constant $C(\phi,\psi)>0$ such that
$$\left|{\rm Cor}(\phi\circ f^n,\psi)\right|\leq C(\phi,\psi)r^n\ \text{ for every $n\geq1$};$$

\item[3.] the geometric pressure function  $t\in(t^-,t^+)\mapsto\mathscr{P}(t)$ is analytic;

 \item[4.] $t^+\leq1$. If $t>t^+$ then $\mathscr{P}(t)=0$ and there exists no equilibrium 
state for the potential $-t\log|Df|$.
\end{itemize}

\end{theorema}

A measure $\mu\in\mathcal M(f)$ satisfies the Central Limit Theorem if given 
a H\"older continuous function $\varphi$ with $\int\varphi d\mu=0$
which is not a coboundary
$(\varphi\neq\psi\circ f-\psi\text{ for any }\psi\in L^2(\mu))$, there exists $\sigma>0$ such that
for any $r\in\mathbb R$,
$$
\lim_{N\to\infty}\mu\left\{x\in \mathbb R\colon\frac{1}{\sqrt{N}}\sum_{n=0}^{N-1}\varphi\circ f^n(x)<r\right\} 
=\frac{1}{\sqrt{2\pi}\sigma}\int_{-\infty}^re^{-\frac{x^2}{2\sigma^2}}dx.
$$

\begin{cor1}
Under the hypotheses of Theorem A, for each $t\in(t^-,t^+)$
the measure $\mu_t$ satisfies the Central Limit Theorem.
\end{cor1}

The existence and uniqueness of the equilibrium state as in Theorem A was proved
by Iommi and Todd  (See \cite[Appendix A]{IomTod10}\footnote{The preclusion of the case
of pre-periodic critical point in \cite{IomTod10} is not essential.}). 
Their method is more general but hides 
statistical properties of the equilibrium states.
We limit ourselves to non-recurrent maps and give more precise descriptions.

By the result of Benedicks and Misiurewicz \cite{BenMis89},
the acim for the map in Theorem A is 
unique up to a multiplicative constant, and is a finite measure if and only if $\int\log|Df(x)|dx>-\infty$. 
If this the case, then the normalization of the acim is denoted by $\mu_{\rm ac}$.
Then $t^+=1$, and the $\mu_{\rm ac}$ is the unique equilibrium state for the potential $-\log|Df|$,
see Proposition \ref{lackexp}. 

For the measure $\mu_{\rm ac}$ one cannot expect
exponential decay of correlations, see Zweim\"uller \cite[Proposition 1]{Zwe05}.
The next theorem links the rate of decay of correlation of $\mu_{\rm ac}$ to the behavior of the function $\ell$.
For two positive functions $a(x)$, $b(x)$ defined around the critical point $c$, the expression $a(x)\sim b(x)$
indicates that $a(x)/b(x)$ is bounded and bounded away from $0$.

\begin{theoremb}
Let $f\colon X\to X$ be a topologically transitive $S$-unimodal map with a flat critical point $c$
satisfying the Misiurewicz condition. Then the following holds for the measure $\mu_{\rm ac}$:

\begin{itemize}

\item[1.] {\rm (Stretched exponential decay)} if $\ell(x)\sim|\log|x-c||^\alpha$ holds for some $\alpha>0$, then 
for every $\tilde\alpha\in(0,\frac{1}{\alpha+1})$, every $\phi\in L^\infty$ and
 every H\"older continuous $\psi\colon X\to\mathbb R$
  there exists a constant $C(\phi,\psi)>0$ such that

$${\rm Cor}(\phi\circ f^n,\psi)\leq
 C(\phi,\psi)e^{-n^{\tilde\alpha}}
\quad\text{for every $n\geq1$};$$

\item[2.] {\rm (Polynomial decay)} if there exists a positive $C^3$ function $v$ on $X$ such that
$v(c)\in(0,1)$ and
 $\ell(x)=|x-c|^{-v(x)}$ for all $x$ near $c$, then 
for every $\beta\in(0,1/v(c)-1)$, every $\phi\in L^\infty$ and
 every H\"older continuous $\psi\colon X\to\mathbb R$
 there exists a constant $C(\phi,\psi)>0$ such that
  $${\rm Cor}(\phi\circ f^n,\psi)\leq C(\phi,\psi)n^{-\beta}\quad\text{for every $n\geq1$}.$$
\end{itemize}
\end{theoremb}

\begin{cor2}
Let $f\colon X\to X$ be a topologically transitive $S$-unimodal map with a flat critical point $c$
satisfying the Misiurewicz condition. 
If $\ell(x)\sim|\log|x-c||^\alpha$ for some $\alpha>0$,
or there exists a positive $C^3$ function $v$ on $X$ such that
$v(c)\in(0,1/2)$ and
 $\ell(x)=|x-c|^{-v(x)}$ for all $x$ near $c$,
then the measure $\mu_{\rm ac}$ satisfies the Central Limit Theorem.
\end{cor2}

The next theorem shows one case in which
 $\mu_{\rm ac}$ has sub-exponential decay of correlations.

 \begin{theoremc}
 Let $f\colon X\to X$ be a topologically transitive $S$-unimodal map with a flat critical point $c$
satisfying the Misiurewicz condition such that 
there exists a positive $C^3$ function $v$ on $X$ such that
$v(c)\in(0,1)$ and $\ell(x)=|x-c|^{-v(x)}$ for all $x$ near $c$.
Then the following holds for the measure $\mu_{\rm ac}$: there exist $\phi\in L^\infty$, a H\"older continuous $\psi$ and
 $C(\phi,\psi)>0$ such that
 $${\rm Cor}(\phi\circ f^n,\psi)\geq C(\phi,\psi)n^{1-\frac{1}{v(c)}}\ \ \text{for every $n\geq1$}.$$
 
\end{theoremc}
Many of the existing results on decay of correlations are only concerned with upper bounds of decay rates.
Lower bounds were first obtained by Sarig \cite{Sar02} for interval maps with indifferent fixed point, followed by Gou\"ezel \cite{Gou04} and Hu \cite{Hu04}.
  To our knowledge, Theorem C is the first to have given a lower bound
for one-dimensional maps with critical point.

In view of Theorem A and Theorem C, a natural question to ask is
 if $\mu_t$ converges weakly to $\mu_{\rm ac}$
as $t\nearrow 1$. The answer is affirmative in the following particular case. 



\begin{theoremd}
Let
\footnote{Maps in Theorem D satisfies the Misiurewicz condition.
Indeed, (M1) is obvious. Since the Schwarzian derivative is negative,
all indifferent periodic points must be attracting. The topological transitivity implies
the non-existence of an attracting periodic point and (M2) holds.}
 $f\colon X\to X$ be a topologically transitive
 $S$-unimodal map with a flat critical point $c$ such that 
$f(0)=0=f(1)$ and $f(c)=1$. If
 $\ell(x)\sim|\log|x-c||^\alpha$ for some $\alpha>0$,
 or there exists a positive $C^3$ function $v$ on $X$ such that
$v(c)\in(0,1/2)$ and
 $\ell(x)=|x-c|^{-v(x)}$ 
for all $x$ near $c$,
then $\mu_t$ converges weakly to $\mu_{\rm ac}$ as $t\nearrow 1$.
\end{theoremd}

In \cite{Tak17} an example is given in which $\ell(x)\sim|x-c|^{-1}$, $t^+=1$
and $\mu_t$ converges as $t\nearrow t^+$ to the unique measure supported on $\omega(c)$, denoted by $\delta(c)$.
The acim of this map is an infinite measure and thus 
there is no equilibrium state for the potential $-\log|Df|$.
At present we lack the following examples: 
(i) $t^+<1$ and $\mu_t\to\delta(c)$;  (ii) the acim is a finite measure and $\mu_t\to\delta(c)$;
(iii) the acim is a finite measure and
$\mu_t\to u\delta(c)+(1-u)\mu_{\rm ac}$, $0<u<1$.
A natural candidate for (ii) (iii) would be $\ell(x)\sim|x-c|^{-\gamma}$, $\gamma\in[1/2,1).$


In the remainder of the introduction let us briefly sketch proofs of the theorems. 
Since the map $f$ has a critical point, it fails to be uniformly hyperbolic in a severe way.
In order to control the effect of the critical point, a main tool is {\it an inducing scheme}.
It defines an induced map, denoted by $\widehat f$,
 which turns out to be uniformly expanding (albeit with countably infinite number of branches),
 no matter how flat the critical point is. 
  
 By virtue of the Misiurewicz condition, the construction of the inducing scheme
 is straightforwardly done, by considering the first return map to a small interval $I$
 containing the critical point. Combining the expansion of the inducing scheme
 with the flatness of the critical point, we conclude $t^+\leq1$ (See Proposition \ref{lackexp}).

Following the line of Pesin and Senti \cite{PesSen05,PesSen08}, 
 we use the countably infinite Markov partition associated with the inducing scheme in order to code the dynamics into 
 a countable Markov shift.
We apply the results of Mauldin and Urba\'nski
\cite{MauUrb03} and of Sarig \cite{Sar99} on the thermodynamic formalism for countable Markov shifts,
in order to establish the existence and uniqueness of equilibrium state
for the induced map $\widehat f$.
This measure is a Gibbs measure, and turns out to have exponential tail, which
permits us to apply the result of Young \cite{You98} to show that the resultant $f$-invariant measure has exponential decay of correlations.
 The last step of the proof of Theorem A is to show that this measure coincides with the equilibrium state $\mu_t$
constructed by Iommi and Todd \cite{IomTod10}.
We do this by showing that the latter is liftable to the inducing scheme
(See Proposition \ref{uniquel}). The analyticity of the pressure is a by-product. 

A strategy for a proof of Theorem B is to analyze the first return time $R$ to the interval $I$, 
and estimate the Lebesgue measure of the tail $\{R>n\}$, 
and again apply the results of Young \cite{You99}.
The tail estimate is straightforward if the first return time is monotone increasing. 
To treat the general case, we need to carefully analyze the distribution 
of the first return time, see Proposition \ref{inducep}.
We prove Theorem C by putting together our tail estimate with the result of  Gou\"ezel \cite{Gou04}. 

To prove Theorem D,
we first show that any accumulation point of the sequence of equilibrium states as $t\nearrow1$
is a convex combination of the acip and the Dirac measure at the fixed point $0$.
We then show that the accumulation point is indeed the acip.
One key observation is that two constants in the Gibbs condition \eqref{gibbs} stay bounded and 
bounded away from $0$ as $t\nearrow 1$.

The rest of this paper consists of two sections.
In Sect.2 we develop preliminary constructions and estimates needed for the proofs of the main results.
All the theorems are proved in Sect.3.




\section{preliminaries}
This section is organized as follows.
In Sect.\ref{ischeme} we introduce inducing schemes, and state and prove its basic properties.
In Sect.\ref{s0} we develop analytic estimates associated with the inducing schemes,
and use them to show the occurrence of the freezing phase transition.
In Sect.\ref{analysis} we prove main technical estimates.
\subsection{Basic properties of inducing schemes}\label{ischeme}
Let $f$ be a unimodal map.  
Let $U$ be a subset of $X$ and $n\geq1$ an integer. Each connected component of $f^{-n}(U)$ is called \emph{a pull-back of $U$ by $f^n$},
or simply {\it a pull-back of $U$.}
A pull-back~$J$ of~$U$ by~$f^n$ is \emph{diffeomorphic} if $f^n
\colon J \to U$ is a diffeomorphism.
An open subinterval $I$ of $X$ is {\it nice} if $f^n(\partial I)\cap I=\emptyset$ holds for every $n\geq1$.

Let $f$ be a unimodal map satisfying the Misiurewicz condition.
Let $I$ be a nice interval which contains the critical point $c$ of $f$. 
Diffeomorphic pull-backs of $I$ are mutually disjoint.
Assume there exist $\tau>0$
and a closed subinterval $\widehat I$ of $X$ which contains the concentric open interval with $I$
which has length $(1+2\tau)|I|$ 
and satisfies 
$f^n(c)\notin \widehat  I$ for every $n\geq1$.
Put $K_\tau=(1+\tau)^2/\tau$.
Then every pull-back of $I$ is diffeomorphic. 
If $W$ is a pull-back of $I$, 
the integer $r\geq1$
such that $f^r(W)=I$ is unique.
The pull-back $W$ of $I$ is {\it primitive} if 
$f^k(W)\cap I=\emptyset$ holds for every $k\in\{0,\ldots, r-1\}\setminus\{0\}$.

{\it The first return time to $I$} is a function $R\colon I\to \mathbb Z_{>0}\cup\{+\infty\}$ 
defined by
$$R(x)=\inf\left(\{n>0\colon f^n(x)\in I\}\cup\{+\infty\}\right).$$
 If $W$ is a primitive pull-back of $I$, then $R$ is constant on $W$ and this common value is denoted by $R(W)$.
Let $\mathcal W$ denote the collection of all primitive pull-backs of $I$ which are contained in $I$.
The triplet $(I,\mathcal{W},R)$ is called
{\it an inducing scheme}.
  Define an induced map
$\widehat f\colon \bigcup_{J\in\mathcal W} J\to I$ by $\widehat f(x)=f^{R(J_x)}(x)$, where $J_x$ denotes the element of $\mathcal W$ containing $x$.

Below we state and prove basic properties of inducing schemes.
 
 \subsubsection*{\underline{Liftability}}
Consider the dynamical system
 on $\bigcap_{n\geq0}({\widehat f})^{-n}\left(\bigcup_{J\in\mathcal W} J\right)$ generated by $\widehat f$, and
 let $\mathcal M(\widehat f)$ denote the set of $\widehat f$-invariant Borel probability measures.
For a measure $\widehat\mu\in\mathcal M(\widehat f)$ 
for which $\int Rd\widehat\mu$ is finite,
define $$\mathcal L(\widehat\mu)=\frac{1}{\int R d\widehat\mu}\sum_{J\in\mathcal W}\sum_{n=0}^{R(J)-1}(f^n)_*(\widehat \mu|_{J}).$$
It is straightforward to check that 
$\mathcal L(\widehat\mu)\in\mathcal M(f)$. 
A measure $\mu\in\mathcal M(f)$ is {\it liftable to the inducing scheme $(I,\mathcal W,R)$} if
there exists $\widehat\mu\in\mathcal M(\widehat f)$ such that $\int R d\widehat\mu$ is finite and $\mathcal L(\widehat\mu)=\mu$.
Not all measures are liftable. For instance measures whose supports are contained in $\omega(c)$
are not liftable.

\begin{lemma}\label{liftlem}
Let $f$ be a unimodal map satisfying the Misiurewicz condition,
and let  $(I,\mathcal W,R)$ be an inducing scheme.
If $\mu\in\mathcal M(f)$ and $\mu(I)>0$, then $\mu$ is liftable to $(I,\mathcal W,R)$.
\end{lemma}
\begin{proof}
Since $R$ is the first return time to $I$, 
if $\mu\in\mathcal M(f)$ and $\mu(I)>0$, then $\int Rd\mu$ is finite (in fact, equal to $1$, see Kac \cite{Kac47}).
 From the result of Zweim\"uller \cite{Zwe05}, $\mu$ is liftable.
\end{proof}

\subsubsection*{\underline{Expansion and bounded distortion}}
For $x,y\in \bigcap_{n\geq0} (\widehat f)^{-n}(\bigcup_{J\in\mathcal W} J)$, $x\neq y$
define $$s(x,y)=\min\{k\geq0\colon \text{$({\widehat f})^k(x)$ and $(\widehat f)^k(y)$ belong to different elements of $\mathcal W$}\}.$$
If $f$ is $S$-unimodal, then
the Koebe Principle \cite[p.277 Theorem 1.2]{dMevSt93} gives
\begin{equation*}\frac{|(D\widehat f)^n(y)|}{|(D\widehat f)^n(x)|}\leq K_\tau\ \ \text{for every }n\in\{0,1,\ldots, s(x,y)\}.\end{equation*}
From this and Ma\~n\'e's hyperbolicity theorem \cite[Theorem A]{Man85}
there exist $\lambda>1$ and an integer $m\geq1$ such that 
\begin{equation*}
|(D\widehat f)^m(x)|\geq\lambda\quad \text{for all $x\in \bigcap_{n=0}^{m-1} (\widehat f)^{-n}\left(\bigcup_{J\in\mathcal W} J\right)$}.\end{equation*}
From these and the last part of
the Koebe Principle \cite[p.277 Theorem 1.2]{dMevSt93} there exists a constant $C>0$ such that
if $s(x,y)\geq1$ then
$$\log\frac{|D\widehat f(x)|}{ |D\widehat f(y)|}\leq C|\widehat f(x)-\widehat f(y)|.$$
Since
$1\geq|I|\geq |(\widehat f)^{s(x,y)}(y)-(\widehat f)^{s(x,y)}(x)|\geq\lambda^{s(x,y)-1}|\widehat f(x)-\widehat f(y)|$
we obtain
\begin{equation}\label{holder}
\log\frac{|D\widehat f(x)|}{ |D\widehat f(y)|}\leq C\lambda^{-s(x,y)+1}.\end{equation}

\subsection{Detecting freezing phase transition}\label{s0}
For $\mu\in\mathcal M(f)$ put $$\chi(\mu)=\int\log|Df|d\mu$$ and call this number {\it a Lyapunov exponent} of $\mu$.

\begin{lemma}\label{positive}
Let $f$ be an $S$-unimodal map satisfying the Misiurewicz condition.
 Then $\chi(\mu)>0$ holds for every $\mu\in\mathcal M(f)$.
\end{lemma}
\begin{proof}
From Ma\~n\'e's hyperbolicity theorem
and Birkhoff's ergodic theorem,
$\chi(\mu)>0$ holds for every ergodic $\mu\in\mathcal M(f)$ whose support 
does not contain $c$.
Let $\mu\in\mathcal M(f)$ be ergodic whose support contains $c$.
Poincar\'e's recurrence theorem implies $\mu(I)>0$, and hence $\mu$ is liftable
(See Sect.\ref{ischeme}). From \cite[Theorem 3]{BruTod09}, $\chi(\mu)>0$ holds.
From the ergodic decomposition theorem, the positivity also holds
for non-ergodic measures.
\end{proof}

To detect the freezing phase transition we need the next analytic estimates.

\begin{lemma}\label{P}
Let $f$ be a unimodal map of class $C^2$ satisfying the Misiurewicz condition,
and let  $(I,\mathcal W,R)$ be an inducing scheme.
Then
   the following holds:
$$  R(x)\sim\ell(x)\log|x-c|^{-1};$$
$$ |Df^{R(x)}(x)|\sim\left|D\ell(x)\log|x-c|+\frac{\ell(x)}{x-c}\right|.$$
\end{lemma}
\begin{proof}
Let $x\in I$.
Since $f$ is $C^2$ on $X\setminus\{c\}$ and Ma\~n\'e's hyperbolicity theorem, there exists a constant $C=C(I)\geq1$
 such that for every $z\in X$ in between $f(x)$ and $f(c)$,
 \begin{equation}\label{subeq-1}C^{-1}\leq \frac{|Df^{R(x)-1}(z)|}{|Df^{R(x)-1}(f(c))|}\leq C.\end{equation}
By (F2), up to $C^1$ changes of coordinates around $c$ and $f(c)$ we have
$f(x)=f(c)-|x-c|^{\ell(x)}$ for all $X\setminus\{c\}$ near $c$.
From this and \eqref{subeq-1},
\begin{equation}\label{criexp}
|f^{R(x)}(x)-f^{R(x)}(c)|\sim|x-c|^{\ell(x)}|Df^{R(x)-1}(f(c))|.\end{equation}
There exist constants $C\geq1$ and $0<\lambda_0<\lambda_1$ such that
\begin{equation}\label{critexp}
C^{-1}e^{\lambda_0(R(x)-1)}\leq |Df^{R(x)-1}(f(c))|\leq Ce^{\lambda_1(R(x)-1)},\end{equation}
and therefore there exists a constant $\tilde C=\tilde C(I)>1$ such that
\begin{equation}\label{subeq0}
\tilde C^{-1}|x-c|^{\ell(x)}e^{\lambda_0(R(x)-1)}\leq |f^{R(x)}(x)-f^{R(x)}(c)|\leq \tilde C|x-c|^{\ell(x)}e^{\lambda_1(R(x)-1)}.\end{equation}
Since $f^{R(x)}(x)\in I$ and $f^{R(x)}(c)\notin\widehat I$,
$\tau|I|\leq |f^{R(x)}(x)-f^{R(x)}(c)|$ holds. Plugging this into the second inequality in 
\eqref{subeq0} gives a lower estimate of $R(x)$.
Plugging $|f^{R(x)}(x)-f^{R(x)}(c)|\leq1$ into the first inequality in 
\eqref{subeq0} gives an upper estimate of $R(x)$.
These two estimates together imply the desired one.

For the derivative estimate, note that \begin{align*}
|Df(x)|&\sim |x-c|^{\ell(x)}\left|D\ell(x)\log|x-c|+\frac{\ell(x)}{x-c}\right|\\
&\sim |f(x)-f(c)|\left|D\ell(x)\log|x-c|+\frac{\ell(x)}{x-c}\right|,
\end{align*}
and thus
\begin{align*}
|Df^{R(x)}(x)|\sim  |Df^{R(x)-1}(f(x))|     |f(x)-f(c)|\left|D\ell(x)\log|x-c|+\frac{\ell(x)}{x-c}\right|.
\end{align*}
For some $z$ in between $f(x)$ and $f(c)$, $|f^{R(x)}(x)-f^{R(x)}(c)|=|Df^{R(x)-1}(z)||f(x)-f(c)|$ holds.
From this and \eqref{subeq-1} we obtain
$|Df^{R(x)-1}(f(x))||f(x)-f(c)|\sim |(f^{R(x)})(x)-f^{R(x)}(c)|.$
\end{proof}

\begin{lemma}\label{lack}
Let $f$ be an $S$-unimodal map with flat critical point
   satisfying the Misiurewicz condition. Then  $\chi_{\rm inf}=0$.
\end{lemma}
\begin{proof}
Taking renormalizations if necessary, we may assume $f$ is topologically mixing.
We construct a sequence of periodic points of $f$
 whose Lyapunov exponents converge to $0$.

We claim
\begin{equation*}\label{asym}
\displaystyle{\liminf_{x\to c-0}}\frac{\log|D\ell(x)|}{\ell(x)\log|x-c|^{-1}}=0.
\end{equation*}
If this is not the case, then from $D\ell>0$ on $[0,c)$ there exists $\varepsilon>0$ such that for all $x\in[0,c)$ near $c$,
$$D\ell(x)e^{-\varepsilon\ell(x)}\geq-\frac{1}{x-c}.$$
Let $x_0\in[0,c)$ be sufficiently near $c$.
Integrating the above inequality from $x_0$ to $x\in(x_0,c)$,
$$e^{-\varepsilon\ell(x)}\leq e^{-\varepsilon\ell(x_0)}+ \varepsilon\log\frac{|x-c|}{|x_0-c|}.$$
The right-hand-side goes to $0$ as 
 $x\to c-|x_0-c|\exp\left(-\varepsilon^{-1}e^{-\varepsilon\ell(x_0)}\right),$
and hence $\ell(x)\to+\infty$ in the same limit.
This yields a contradiction.

Let $(I,\mathcal W,R)$ be an inducing scheme. Since $f$ is topologically mixing,
 $\overline{I}\setminus\{c\}=\bigcup_{J\in\mathcal W} \overline{J}$ holds.
Since each $J\in\mathcal W$ contains a periodic point of least period $R(J)$, it is possible to take
a monotone decreasing sequence
 $\{x_n\}_{n\geq0}$ of periodic points converging to $c$
 for which 
$$ \displaystyle{\lim_{n\to\infty}}\frac{\log|D\ell(x_n)|}{\ell(x_n)\log|x_n-c|^{-1}}=0.$$
Let $J_{x_n}$ denote the element of $\mathcal W$ which contains $x_n$.
The estimates in Lemma \ref{P} gives
\begin{equation*}\lim_{n\to\infty}\frac{1}{R(J_{x_n})}\log|Df^{R(J_{x_n})}(x_n)|=0.\qedhere\end{equation*}
\end{proof}

We are in position to show the occurrence of the freezing phase transition.
\begin{prop}\label{lackexp}
   Let $f$ be an $S$-unimodal map with flat critical point
   satisfying the Misiurewicz condition. Then $t^+\leq1$ holds.
 If the absolutely continuous invariant measure is a finite measure, then $t^+=1$, and $\mu_{\rm ac}$ is the unique equilibrium state
for the potential $-\log|Df|$. 
  \end{prop}

\begin{proof}
Ruelle's inequality \cite{Rue78} implies $\mathscr{P}(1)\leq0$. 
From Lemma \ref{lack}, $\mathscr{P}(t)>0$ holds for $t<t^+$, and thus
 $t^+\leq1$. 
From Lemma \ref{positive}, a measure $\mu\in\mathcal M(f)$ is absolutely continuous with respect to the Lebesgue measure
if and only if $h(\mu)-\chi(\mu)=0$ 
 \cite[Theorem 1.5]{Dob14}.
 If $t^+<1$, then for every $t\in(t^+,1)$ we have $\mathscr{P}(t)\geq h(\mu)-t\chi(\mu)>0$, a contradiction.
 The last statement of Proposition \ref{lackexp} is a consequence of  \cite[Theorem 1.5]{Dob14}.
\end{proof}

\subsection{Technical estimates}\label{analysis}
Given an inducing scheme $(I,\mathcal W, R)$, 
for each $n>1$ put $$S(n)=\{J\in\mathcal W\colon R(J)=n\}$$ and
 $$\{R>n\}=\{x\in I\colon R(x)>n\}.$$
 The Lebesgue measure of a subset $A$ of $X$ is denoted by $|A|$.

\begin{prop}\label{inducep}
Let $f$ be a topologically transitive $S$-unimodal map 
satisfying the Misiurewicz condition.
There exists an inducing scheme $(I,\mathcal W, R)$ with the following properties:

\begin{itemize}


\item[1.] 
for any $\gamma>1$ there exists a constant $c_\gamma>0$ 
such that $$\#S(n)<c_\gamma\gamma^n\quad\text{for every }n\geq1;$$

\item[2.] if the critical point $c$ is flat and $\ell(x)\sim|\log|x-c||^\alpha$ for some $\alpha>0$,
then there exist constants $C_0>0$ and $C_1>0$ such that 
$$|\{R>n\}|\leq C_0e^{-C_1n^{\frac{1}{\alpha+1}}}\quad\text{for every }n\geq1;$$

\item[3.] if there exists a positive $C^3$ function on $X$ such that $v(c)\in(0,1)$
and $\ell(x)=|x-c|^{-v(x)}$ for all $x$ near $c$,
then  for every $\beta>v(c)$ there exist constants $0<C_2<C_3$
such that
 \begin{equation*}
C_2n^{-\frac{1}{v(c)}}\leq|\{R>n\}|\leq C_3n^{-\frac{1}{\beta}}\quad\text{for every }n\geq1.\end{equation*}
\end{itemize}

\end{prop}






The rest of this section is dedicated to 
 a proof of Proposition \ref{inducep}. To this end we need a few preliminary considerations.
 \medskip
 
\noindent{\it Notation.}
For two subsets $A$, $B$ of $X$,
write $A<B$ if $\sup A<\inf B$.
For a point $x^+\in X$ with $x^+>c$,
   $x^-$ denotes the point in $X\setminus\{x^+\}$
such that $f(x^-)=f(x^+)$.
For a subset $A^+$ of $X$ with $A^+>\{c\}$ denote $A^-=\{x^-\in X\colon x^+\in A^+\}$.

 \begin{lemma}\label{lemmore}
 Under the assumptions of Proposition \ref{inducep},
 let $(I,\mathcal W,R)$ be an inducing scheme and
let $W_1$, $W_2$ be distinct primitive pull-backs of $I$ such that $R(W_1)=R(W_2)$.
There exists a primitive pull-back $W$ of $I$ such that $W_1<W<W_2$ and $R(W)<R(W_1)$.
\end{lemma}

\begin{proof}
Put $m=R(W_1)=R(W_2)$.
Let $U$ denote the minimal open interval containing $W_1$ and $W_2$.
Let $n>0$ be the smallest integer such that $c\in f^n(U)$. We must have $n<m$. 
Since $f^k(W_1\cup W_2)\cap I=\emptyset$
for every $k\in\{1,2,\ldots,n\}$,  $f^k(U)\cap I=\emptyset$
for every $k\in\{1,2,\ldots,n-1\}$.
Since $W_1$, $W_2$ are primitive, $f^n(W_1\cup W_2)\cap I=\emptyset$.
Define $W$ to be the pull-back of $I$ by $f^n$ which is contained in $U$. 
\end{proof}

Let $(I,\mathcal W,R)$ be an inducing scheme and
$U$ be a subset of $X$. A subset $W$ of $U$ is {\it the minimal pull-back of $I$ in $U$} 
if it is a primitive pull-back of $I$ and for any other primitive pull-back $W'$ of $I$ which is contained in $U$,
  $R(W')> R(W)$ holds.


   \begin{proof}[Proof of Proposition \ref{inducep}]
   Since $f$ is topologically transitive, it is topologically mixing and
       $c$ is accumulated by periodic points from both sides.
There exist a nice interval $I=(a_0^-,a_0^+)$ 
       which satisfies $\overline{I}\cap\overline{\{f^n(c)\colon n\geq1\}}=\emptyset$, and
         $f^{R_0}(a_0^-)=a_0^-$ or  $f^{R_0}(a_0^+)=a_0^+$ where $R_0=\min\{n>0\colon f^n(I)\cap I\neq\emptyset\}$.     
         Without loss of generality we may assume $f^{R_0}(a_0^+)=a_0^+$.
Let $\mathcal W$ denote the collection of primitive pull-backs of $I$ and $R$ the first return time to $I$.
     We show that the inducing scheme $(I,\mathcal W,R)$ satisfies the desired properties.

Put
$$\eta=\min\{|f^n(a_0^+)-x|\colon n\geq1,x\in\partial I\}.$$
Since $I$ is a nice interval, $\eta>0$ holds.
    Let $V_0^+$ denote the minimal pull-back of $I$ in $(a_0^+,a_0^++\eta)$.
 Let $i\geq0$ and suppose a sequence $V_0^+,\ldots,V_i^+$ of primitive pull-backs of $I$ has been defined.
From Lemma \ref{lemmore}
 there are two cases: 
 \smallskip
 
 (i) for any primitive pull-back $W$ of $I$ such that $V_i^+<W$,
$R(V_i^+)<R(W)$; 

(ii) there exists a primitive pull-back $W$ of $I$ such that $V_i^+<W$ and $R(V_i^+)>R(W)$.
\smallskip

In case (i) we stop the construction. In case (ii) 
write $V_i^+=(a,b)$ and define $V_{i+1}^+$ to be the minimal pull-back of $I$ 
in $(b,1)$. 
This construction stops in finite time and we end up with
a sequence $\{V_i^+\}_{i=0}^{N}$ of primitive pull-backs of $I$
with the following properties:

\begin{itemize}
\item[1.] $I<V_0^+<V_1^+<\cdots<V_N^+$; 

\item[2.] $R(V_0^+)>R(V_1^+)>\cdots>R(V_{N}^+)=1$;

\item[3.] if $\partial V_0^+\cap\partial I=\emptyset$ and  $W$ is a primitive pull-back of $I$ such that
 $I<W<V_0^+$, then $R(W)>R(V_0^+)$;

\item[4.] If $W$ is a primitive pull-back of $I$ such that $V_{i}^+<W<V_{i+1}^+$ holds
for some $i\in\{0,\ldots, N-1\}$, then $R(W)> R(V_{i}^+)$. 
\end{itemize}

We construct
a sequence $\{J_k^+\}_{k\geq0}$ of open subintervals of $I$,
 a sequence $\{R_k\}_{k\geq0}$ of positive integers inductively as follows.
Define $J_0^+$ to be the pull-back of $I$ by $f^{R_0}$ which is contained in $I$ and 
satisfies $\{c\}<J_0^+$. 
Now, let $k\geq0$ and suppose we have defined
$J_0^+,J_1^+,\ldots,J_{k}^+$ and $R_0,R_1,\ldots,R_k$
with the desired properties 
such that the set
$ f^{R_k}\left(I\setminus(\bigcup_{0\leq n\leq k}  J_n^+)\right)$
does not intersect $I$. Since this set contains $(a_0^--\eta,a_0^-)$ or $(a_0^+,a_0^++\eta)$,
 the minimal pull-back of $I$ in this set, denoted by
$W$, belongs to $\{V_{i}^-\}_{0\leq i\leq N}\cup\{V_{i}^+\}_{0\leq i\leq N}$.
Let $J$ denote the pull-back of $I$ by $f^{R_k+R(W)}$ which is 
contained in $I\setminus(\bigcup_{0\leq n\leq k} J_n^-\cup J_n^+)$ and
satisfies $\{c\}<J$. 
If $\partial J_k^+\cap\partial J\neq\emptyset$, then set $J_{k+1}^+=J$
and $R_{k+1}=R_k+R(W)$.
Otherwise, set $J_{k+2}^+=J$ and
define $J_{k+1}^+$ to be the maximal open interval sandwiched by $J_{k}^+$ and $J_{k+2}^+$.
Set $R_{i}=R_{k}+R(W)$ for $i\in\{k+1,k+2\}$.

From the construction the following holds:

\begin{itemize}

\item[1.] $\{c\}<\cdots<J_{k+1}^+<J_k^+<\cdots<J_0^+$;


\item[2.] 
for each $k\geq0$ one of the following holds:

\begin{itemize}
\item[(i)] $J_{k}^+\in \mathcal W$ 
and $R(J_k^+)=R_k$;

\item[(ii)]  $J_{k}^+\notin \mathcal W$,
 $I\cap f^i(J_{k}^+)=\emptyset$ for every $i\in\{1,2,\ldots,R_{k}\}$,
$\partial I\cap\partial f^{R_{k}}(J_{k}^+)\neq\emptyset$;

\end{itemize}

\item[3.] If $J_{k+1}^+\notin\mathcal W$, then $J_{k}^+\in\mathcal W$
and $R_{k}=R_{k+1}$.

\end{itemize}
It follows that for every $k\geq0$,
\begin{equation}\label{R1}0\leq R_{k+1}-R_{k}\leq R(V_0^+)\quad\text{and}\quad
\frac{k}{2}\leq R_k-R_0\leq R(V_0^+)k.\end{equation}

\begin{remark}\label{chebyshev}
{\rm In the case $f(c)=1$ and $f(0)=0=f(1)$ like the Chebyshev quadratic $f(x)=4x(1-x)$, 
we choose $I=(a_0^-,a_0^+)$ where $a_0^+$ is the fixed point which is not $0$.
Let $b^-<c$ be such that $f(b^-)=a_0^-$. We have
$V_0^-=(b^-,a_0^-)$, $N=0$ and
 $R_k=k+2$ for every $k\geq0$.
In particular, for every $k\geq0$, $J_k^\pm$ are primitive pull-backs of $I$ with $R(J_k^\pm)=k+2$.}
\end{remark}

\medskip

\noindent{\it Proof of Proposition \ref{inducep}-1.}
Let $\gamma>1$ be such that $\gamma(3-\gamma)\geq2.$
Put $c_\gamma=2\gamma^{-R_0}$.
Then $\#S(R_0)=2=c_\gamma\gamma^{R_0}$ holds.
We argue by induction. Let $m\geq R_0$ and suppose that 
$\#S(n)\leq c_\gamma\gamma^n$ holds for every $n\in\{R_0,\ldots,m\}$.

If $J\in\mathcal W$ and $R(J)=n$, then there exist $\tilde J\in\mathcal W$, $\tilde n\in[n-q,n-1]$ such that
$R(\tilde J)=\tilde n$ and a pull-back $\tilde W$ of $I$ by $f^{\tilde n}$ which contains both $J$ and $ \tilde J$.
The triplet $(\tilde J,\tilde n,\tilde W)$ with this property is unique for each $J$, and 
 the correspondence 
$J\in S(n)\mapsto (\tilde J,\tilde n,\tilde W)$ is at most two-to-one.
This yields
$$\#S(m+1)\leq2\sum_{k=1}^{m+1-R_0}\#S(m+1-k).$$
The condition on $\gamma$ implies 
$\gamma^{m+1}(3-\gamma)\geq2\gamma^{R_0}$,
and therefore
\begin{align*}\#S(m+1)&\leq 2\sum_{k=1}^{m+1-R_0}c_\gamma\gamma^{m+1-k}\\
&=2c_\gamma\frac{1-\gamma^{m+1-R_0}}{1-\gamma}\gamma^{R_0}\\
&\leq c_\gamma\gamma^{m+1}.\end{align*}

\noindent{\it Proof of Proposition \ref{inducep}-2.}
Write $J_k^+=(a_{k+1}^+,a_k^+)$ and
put $I_k=[a_k^-,a_k^+]$.
For each large integer $n$ fix an integer $\bar n\geq1$ such that $n/3\leq R_{\bar n}\leq n/2$.
We have $$\{R>n\}\subset \{R>n\}\cap  (I\setminus I_{\bar n})\cup I_{\bar n}.$$
We show that $|\{R>n\}\cap  (I\setminus I_{\bar n})|$ decays exponentially in $n$.
For each integer $k\in\{0,1,\ldots,n\}$
let $\mathscr{A}_k$ denote the collection of connected components of 
$\{R>k\}\cap  (I\setminus I_{\bar n})$,
and
 $\mathscr{A}_k''$ the collection of elements of $\mathscr{A}_k$
whose boundaries contain $a_{\bar n}^+$ or $a_{\bar n}^-$.
Put $\mathscr{A}_k'=\mathscr{A}_k\setminus \mathscr{A}_k''$,
 $|\mathscr{A}_k'|=\sum_{A\in\mathscr{A}_k'}|A|$ and $|\mathscr{A}_k''|=\sum_{A\in\mathscr{A}_k''}|A|$.
Obviously $\#\mathscr{A}_k''=2$ holds for every $k\geq0$. From Ma\~n\'e's hyperbolicity theorem,
$|\mathscr{A}_n''|$ decays exponentially in $n$.

For each $A\in\mathscr{A}_k'$ put $k_A=\min\{m>k\colon \{R=m\}\cap A\neq\emptyset\}$.
\eqref{R1} gives $k_A\leq k+R(V_0^+)$ for every $k\geq R_{\bar n}$.
Since $a_0^+$ is a periodic point, it is possible to choose
 $\tau'\in(0,1)$ so that for each $k\geq0$ and each $A\in\mathscr{A}_k'$ there exists a subinterval of  $X$ on which
$f^{k_A}$ is a diffeomorphism and  the image contains the concentric open interval with $f^{k_A}(A)$ of length $(1+2\tau')|f^{k_A}(A)|$.
Put $\theta=1-\left(\frac{\tau'}{1+\tau'}\right)^2|I|\in(0,1)$.
By the Koebe Principle and $|f^{k_A}(A)|\leq1$,
\begin{align*}\frac{|A\setminus\mathscr{A}_{k_A}'|}{|A|}&=\frac{|\{x\in A\colon f^{k_A}(x)\in I\}|}{|A|}\\
&\geq\left(\frac{\tau'}{1+\tau'}\right)^2\frac{|I|}{|f^{k_A}(A)|}\\
&\geq1-\theta,\end{align*}
and thus
$|A\cap\mathscr{A}_{k_A}'|\leq\theta|A|$. By the monotonicity, $|A\cap\mathscr{A}_{k+R(V_0)^+}'|\leq\theta|A|$ holds.
Summing this over all $A\in\mathscr{A}_k'$ we obtain 
$|\mathscr{A}_{k+R(V_0^+)}'|\leq\theta|\mathscr{A}_k'|$.
Since $R_{\bar n}\leq n/2$ and therefore
$|\mathscr{A}_{n}'|\leq\theta^{[2n/R(V_0^+)]}|$
where $[$ $\cdot$ $]$ denotes the integer part.



We now estimate $|I_{\bar n}|$.
By \eqref{criexp} there exists a constant $C\geq1$ such that
 $C^{-1}\leq|a_{\bar n}^+-c|^{\ell(a_{\bar n}^+)}|Df^{R_{\bar n}-1}(f(c))|\leq C$.
From this and  \eqref{critexp}, for sufficiently large $n$
we have 
\begin{equation}\label{twice}
-\lambda_1n\leq\ell(a_{\bar n}^+)\log|a_{\bar n}^+-c|\leq -\frac{\lambda_0n}{4}.
\end{equation}
This remains true for $a_{\bar n}^-$ instead of $a_{\bar n}^+$.
 From \eqref{twice} and the assumption on $\ell$
 there exists $C>0$ such that for sufficiently large $n$,
\begin{equation}\label{in'}|I_{\bar n}|\leq e^{-(Cn)^{\frac{1}{\alpha+1}}}.\end{equation}
We obtain the desired stretched-exponential decay.
\medskip

\noindent{\it Proof of Proposition \ref{inducep}-3.}
 From \eqref{twice} and the assumption on $\ell$, for every $\beta>v(c)$ there exists $n'\geq1$ such that 
 for every $n\geq n'$,
  \begin{equation}\label{in}
\lambda_1^{-\frac{1}{v(c)}}n^{-\frac{1}{v(c)}}\leq|I_{\bar n}|\leq       \left (\frac{\lambda_0}{4}\right)^{-\frac{1}{\beta}}n^{-\frac{1}{\beta}}.
\end{equation}
We obtain the desired polynomial decay.
 \end{proof}

  \section{Proofs of the main results}
In this section we complete the proofs of the theorems.
The first three sections are dedicated to the proof of Theorem A.
In Sect.\ref{s1} we summarize the results from \cite{MauUrb03,Sar99}
needed for the thermodynamic formalism for the induced map.
In Sect.\ref{s00} we show the existence and uniqueness of an equilibrium state among measures
which are liftable to the inducing scheme.
In Sect.\ref{s11} we  eliminate the possibility 
that non-liftable measures attain the supremum in $\mathscr{P}(t)$,
and complete the proof of Theorem A.
In Sect.\ref{bc} we prove Theorem B and Theorem C.
We prove Theorem D in 
Sect.\ref{lay} and Sect.\ref{flimit}.

Throughout this section, let $f$ be a topologically transitive $S$-unimodal map with a flat critical point $c$ satisfying the Misiurewicz condition.
Fix an inducing scheme $(I,\mathcal W, R)$ for which the conclusions of Proposition \ref{inducep}
hold. Let $\widehat f$ denote the associated induced map.

\subsection{Thermodynamic formalism for the induced map}\label{s1}
 {\it A word of length} $n\geq1$ is a string of $n$-elements of $\mathcal{W}$.
   The set of words of length $n$ is denoted by $W_n$.
  For each $J_0J_1\cdots J_{n-1}\in W_n$
put $R_{J_0J_1\cdots J_{n-1}}=R(J_0)+R(J_1)+\cdots+R(J_{n-1}),$ and
 let $\omega_{J_0J_1\cdots J_{n-1}}$ denote the diffeomorphic pull-back 
of $J$ by the composition $(f^{R(J_{n-1})}|_{J_{n-1}})\circ \cdots\circ (f^{R(J_1)}|_{J_1})\circ (f^{R(J_0)}|_{J_0}).$

For each $t,p\in\mathbb R$, $n\geq1$ and $J\in\mathcal W$ define
$$Z_n(t,p;J)=\sum_{\stackrel{J_0\ldots J_{n-1}\in W_n}{J_{0}=J}}e^{-R_{J_0\cdots J_{n-1}}p}
\sup_{x\in\omega_{J_0\cdots J_{n-1}}}|(D\widehat f)^n(x)|^{-t}.$$
\eqref{holder} implies that the function 
$\varphi=-t\log|D\widehat f|-pR$ defines a weakly H\"older continuous potential on the countable infinite symbolic space
associated with $\widehat f$.
Hence the limit
$$\mathscr{P}(t,p)=\lim_{n\to\infty}\frac{1}{n}\log Z_n(t,p;J)$$
exists (including $+\infty$), which is never $-\infty$ and independent of $J\in\mathcal W$ \cite[Theorem 1]{Sar99}.

Assume that $\mathscr{P}(t,p)$ is finite. 
Then the Variational Principle \cite{Sar99} holds:
$$\mathscr{P}(t,p)=\sup\left\{h_{\widehat \mu}(\widehat f)+\int \varphi d\widehat\mu\colon\widehat\mu\in\mathcal M(\widehat f), \int \varphi d\widehat\mu>-\infty\right\}.$$
Here, $h_{\widehat\mu}(\widehat f)$ denotes the Kolmogorov-Sina{\u\i}  entropy of $(\widehat f,\widehat\mu)$.
There exists a unique measure, denoted by $\widehat\mu_{*}$, 
 which attains this supremum. The $\widehat\mu_*$ is {\it a Gibbs measure}: there exist constants $C_1,C_2>0$ such that 
for every $n\geq1$ and every $J_0\cdots J_{n-1}\in W_n$,
\begin{equation}\label{gibbs}C_1\leq\frac{\widehat\mu_*(\omega_{J_0\cdots J_{n-1}})}{e^{-\mathscr{P}(t,p)n}e^{-R_{J_0\cdots J_{n-1}}p}
\sup_{x\in\omega_{J_0\cdots J_{n-1}}}|(D\widehat f)^n(x)|^{-t}}\leq C_2.\end{equation}

\subsection{Construction of equilibrium states among liftable measures}\label{s00}
Let $\mathcal M_L(f)$ denote the set of elements of $\mathcal M(f)$ which are liftable.
For $t\in\mathbb R$ define
$$\mathscr P_L(t)=\sup\left\{h(\mu)-t\int\log|Df|d\mu\colon\mu\in\mathcal M_L(f)\right\}.$$
A measure which attains this supremum is called {\it an equilibrium state for the potential $-t\log|Df|$ among liftable measures}.
Set
 \begin{equation*}t^-_L=\inf\{t\in\mathbb R\colon \mathscr P_L(t)+t\chi_{\rm sup}>0\},\quad
t^+_L=\sup\{t\in\mathbb R\colon \mathscr P_L(t)+t\chi_{\rm inf}>0\}.\end{equation*}
Note that $t^-\leq t^-_L$ and $t^+_L\leq t^+$.

\begin{prop}\label{uniquel}
Let $f$ be a topologically transitive $S$-unimodal map with flat critical point satisfying the Misiurewicz condition.
For every  $t\in(t^-_L,t^+_L)$ there exists a unique equilibrium state for the potential $-t\log|Df|$ among liftable measures.
In addition, $t\in(t^-_L,t^+_L)\mapsto \mathscr P_L(t)$ is analytic.
\end{prop}

We construct the equilibrium state in Proposition \ref{uniquel} from a Gibbs measure for the induced map.
Hence we need to prove the finiteness of $\mathscr{P}(t,p)$.
\begin{lemma}\label{pr}
Let $f$ be a topologically transitive $S$-unimodal map satisfying the Misiurewicz condition.
If $t_0\in(t^-_L,t^+_L)$, then $\mathscr{P}(t,p)$ is finite for all $(t,p)$ in a neighborhood of $(t_0,\mathscr P_L(t_0))$.
\end{lemma}

\begin{proof}
For $t,p\in\mathbb R$ define
$$T(t,p)=\sum_{J\in \mathcal W}e^{-R(J)p}\sup_{x\in J}|D\widehat f(x)|^{-t}.$$
It suffices to show $T(t,p)$ is finite for all $(t,p)$ in a neighborhood of $(t_0,\mathscr P_L(t_0))$.
Indeed, if this is the case then 
\begin{align*}\mathscr{P}(t,p)&=\lim_{n\to\infty}
\frac{1}{n}\log  \sum_{\stackrel{J_0\ldots J_{n-1}\in W_n}{J_{0}=J}}e^{-R_{J_0\cdots J_{n-1}}p}
\sup_{x\in\omega_{J_0\cdots J_{n-1}}}|(D\widehat f)^n(x)|^{-t}\\
&\leq\lim_{n\to\infty}
\frac{1}{n}\log\left(\sum_{J\in \mathcal W}e^{-R(J)p}\sup_{x\in J}|D\widehat f(x)|^{-t}\right)^n\\
&=\log T(t,p).\end{align*}
Note that
$$T(t,p)=\sum_{n>0}\sum_{J\in \mathcal W,R(J)=n}e^{-R(J)p}\sup_{x\in J}|D\widehat f(x)|^{-t}.$$
Since each $J\in\mathcal W$ contains a periodic point of period $R(J)$, from the bounded distortion
there exist constants $C_1$, $C_2$ such that for every
$J\in \mathcal W$,
\begin{equation}\label{finiteG}
C_1e^{\chi_{\rm inf}R(J)}\leq\sup_{x\in J}|D\widehat f(x)|\leq C_2e^{\chi_{\rm sup}R(J)}.\end{equation}
If $0< t_0< t^+_L,$ then using the first inequality in \eqref{finiteG} and Proposition \ref{inducep}-1,
$$T(t,p)\leq C_1^{-t} \sum_{n>0}\#S(n)e^{-np-t\chi_{\rm inf}n}\leq C_1^{-t}\sum_{n>0}c_\gamma\gamma^ne^{-np-t\chi_{\rm inf}n}.$$
Since $-\mathscr P_L(t_0)-t_0\chi_{\rm inf}<0$, this series converges
for $(t,p)$ sufficiently close to $(t_0,\mathscr P_L(t_0))$ provided $\gamma$ is chosen sufficiently close to $1$.
If $t^-_L< t_0\leq0$ we argue similarly, using the second inequality in \eqref{finiteG} and
 $-\mathscr P_L(t_0)-t_0\chi_{\rm sup}<0$.
\end{proof}

\begin{proof}[Proof of Proposition \ref{uniquel}]
Let $t\in(t^-_L,t^+_L)$.
By Lemma \ref{pr}, $\mathscr P(t,\mathscr P_L(t))$ is finite.
Hence, there exists a unique equilibrium state for the potential
$-t\log|D\widehat f|-\mathscr{P}_L(t)R$, denoted by $\widehat\mu_{t}$.
Then $\int Rd\widehat\mu_t$ is finite, as shown in the next paragraph. Put
$\mu_t=\mathcal L(\widehat\mu_t)$. 
We have
\begin{align*}
\mathscr{P}(t,\mathscr P_L(t))&=h_{\widehat\mu_t}(\widehat f)+\int
-t\log|D\widehat f|-\mathscr P_L(t)Rd\widehat\mu_t\\
&=\int Rd\widehat\mu_t
\left(h_{\mu_t}(f)+\int-t\log|Df|-\mathscr P_L(t)d\mu_t\right)\leq0.
\end{align*}
On the other hand, for every $\varepsilon>0$ there is $\mu\in\mathcal M_L(f)$ such that
$$h_\mu(f)-t\int\log|Df| d\mu\geq \mathscr P_L(t)-\varepsilon.$$
Since $\mu$ is liftable, there exists $\widehat\nu\in\mathcal M(\widehat f)$ such that $\int Rd\widehat \nu$ is finite and $\mathcal L(\widehat\nu)=\mu$.
We have
\begin{align*}\mathscr{P}(t,\mathscr P_L(t)-\varepsilon)&\geq h_{\widehat\nu}(\widehat f)+\int-t\log|D\widehat f|-(\mathscr P_L(t)-\varepsilon)R d\widehat\nu\\
&=
 \int Rd\widehat\nu\left(h_{\mu}(f)+\int-t\log|Df|-\mathscr P_L(t)+\varepsilon d\mu\right)\geq0.\end{align*}
By Lemma \ref{pr}, there exists $\varepsilon_0>0$ such that
$\mathscr P(t,\mathscr P_L(t)-\varepsilon)$ is finite for all $0\leq\varepsilon\leq\varepsilon_0$.
The Variational Principle implies that $\varepsilon\mapsto \mathscr P(t,\mathscr P_L(t)-\varepsilon)$ is convex, and so is continuous.
We conclude that $\mathscr P(t,\mathscr P_L(t))\geq0$, and 
\begin{equation}\label{zero}
0=\mathscr P(t,\mathscr P_L(t))=\int Rd\widehat\mu_t\left(h_{\mu_t}(f)+\int-t\log|Df|-\mathscr P_L(t)d\mu_t\right).
\end{equation}
The measure $\mu_t$ is an equilibrium state for the potential $-t\log|Df|$
among liftable measures. The uniqueness follows from the uniqueness of $\widehat\mu_t$.
\medskip

To show the finiteness of $\int Rd\widehat\mu_t$,
let $J\in\mathcal W$. Since $1\leq \sup_{x\in J}|D\widehat f(x)|\leq (\frac{1+\tau}{\tau})^2e^{\chi_{\rm sup}R(J)}$, 
$\sup_{x\in J}|D\widehat f(x)|^{-t}\leq \max\{e^{-\chi_{\rm sup}R(J)t},1\}$ holds
for all $t\in\mathbb R$.
Since $\widehat\mu_{t}$ is a Gibbs measure and $\mathscr{P}(t,\mathscr P_L(t))=0$,
 \eqref{gibbs} gives
\begin{align*}
\widehat\mu_{t}(J)&\leq C_2e^{-\mathscr P_L(t)R(J)}\sup_{x\in J}|D\widehat f(x)|^{-t}\\
&\leq C_2e^{-\mathscr P_L(t)R(J)} \max\{e^{-\chi_{\rm sup}R(J)t},1\}.\end{align*}
Therefore
\begin{equation*}\sum_{\stackrel{J\in \mathcal W}{R(J)=n}} R(J)\widehat\mu_t(J)\leq
C_2 n \#S(n)e^{-\mathscr P_L(t)n} \max\{e^{-\chi_{\rm sup}nt},1\}.
\end{equation*}
If $t\in(t^-,0)$, then by
 $\mathscr P_L(t)>-t\chi_{\rm sup}$ and Proposition \ref{inducep}-1
 the right-hand-side decays exponentially as $n$ increases.
This is also the case for $t\in[0,t^+)$ by $\mathscr P_L(t)>0$ and Proposition \ref{inducep}-1 again.

It is left to show the analyticity of $t\in(t^-_L,t^+_L)\mapsto\mathscr P_L(t)$.
Let $t_0\in(t^-_L,t^+_L)$.
By Lemma \ref{pr}, $\mathscr{P}(t,p)$ is finite on a neighborhood of $(t,p)=(t_0,\mathscr P_L(t_0))$. 
From \cite[Theorem 2.6.12]{MauUrb03}, $\mathscr{P}(t,p)$ is real-analytic at $(t_0,\mathscr P_L(t_0))$, and so is holomorphic 
at the same point.
\eqref{zero} gives $\mathscr{P}(t_0,\mathscr P_L(t_0))=0$. 
Lemma \ref{delp} below shows
$\frac{\partial\mathscr{P}}{\partial p}(t_0,\mathscr P_L(t_0))\neq0$.
From the Implicit Function Theorem for holomorphic functions of two variables it follows that
$\mathscr P_L(t)$ is analytic at $t=t_0$.
\end{proof}

\begin{lemma}\label{delp}
For every $t\in(t^-,t^+)$,
$\frac{\partial\mathscr{P}}{\partial p}(t,\mathscr P_L(t))<0$.
\end{lemma}
\begin{proof}
For each small $\varepsilon>0$ let $\widehat {\nu}_\varepsilon$ denote the equilibrium state of $\widehat f$ for the potential $-t\log|D\widehat f|-R(\mathscr{P}_L(t)+\varepsilon)$.
Then
\begin{align*}
\mathscr P(t,\mathscr{P}_L(t)+\varepsilon-\varepsilon)
\geq \mathscr{P}(t,\mathscr{P}_L(t)+\varepsilon)+\int R\varepsilon d\widehat\nu_\varepsilon,
\end{align*}
and therefore
$$\frac{1}{\varepsilon}\left(\mathscr{P}(t,\mathscr{P}_L(t)+\varepsilon)- \mathscr{P}(t,\mathscr{P}_L(t))\right)\leq-\int R d\widehat\nu_\varepsilon.$$
It suffices to show that $\int R d\widehat\nu_\varepsilon$ is uniformly positive
for all sufficiently small $\varepsilon$.
Since $\widehat\nu_\varepsilon$ is a Gibbs measure, from \eqref{gibbs}
there exists a constant $C_1(\varepsilon)>0$ such that for every $J\in \mathcal W$,
\begin{equation*}
\widehat\nu_\varepsilon(J)\geq C_1(\varepsilon)e^{-\mathscr{P}(t,
\mathscr{P}_L(t)+\varepsilon)-(\mathscr{P}_L(t)+\varepsilon)R(J)}\sup_{x\in J}|D\widehat f(x)|^{-t}.\end{equation*}
A close inspection of \cite[Proposition 2, Theorem 8]{Sar99} shows that $C_1(\varepsilon)$ can be chosen so that
$ \inf_{\varepsilon}C_1(\varepsilon)>0$, where the infimum is taken over 
all sufficiently small $\varepsilon>0$.
Fix $J^*\in \mathcal W$. We have
\begin{align*}
\int Rd\widehat\nu_\varepsilon&\geq R(J^*)\widehat\nu_\varepsilon(J^*)\\
&\geq  \inf_{\varepsilon}C_1(\varepsilon)R(J^*)e^{-\mathscr{P}(t,\mathscr{P}_L(t)+\varepsilon)-(\mathscr{P}_L(t)+\varepsilon)R(J^*)}
\sup_{x\in J^*}|D\widehat f(x)|^{-t},\end{align*}
which is uniformly positive for all sufficiently small $\varepsilon>0$.
\end{proof}

\subsection{Liftability of the equilibrium states in positive entropy phase}\label{s11}
We are in position to finish the proof of Theorem A.

\begin{proof}[Proof of Theorem A]
By virtue of Proposition \ref{uniquel} 
it suffices to show that $t^-_L=t^-$ and $t^+_L=t^+$, and  the unique equilibrium state $\mu_t$ for the potential $-t\log|Df|$ is liftable 
to the inducing scheme for all  $t\in(t^-,t^+)$.
The statistical properties of $\mu_t$ now follow from \cite{You98}
together with the exponential tail estimate in the proof of Proposition \ref{uniquel}.

For $t\in\mathbb R$
 put $$\mathscr P_{\lnot L}(t)=\sup\left\{h(\mu)-t\int\log|Df|d\mu\colon\mu\in\mathcal M(f)\setminus\mathcal M_L(f)\right\}.$$
It is enough to show that
\begin{equation}\label{press}
\mathscr P_L(t)> \mathscr P_{\lnot L}(t)\ \text{ for every }t\in(t_L^-,t_L^+).\end{equation}
Indeed, from $\mathscr P_L(t_L^-)=-t_L^-\chi_{\rm sup}$ and the continuity
of $t\mapsto\mathscr{P}_L(t)$ and $t\mapsto\mathscr{P}_{\lnot L}(t)$
it follows that
$\mathscr P_{\lnot L}(t_L^-)\leq-t_L^-\chi_{\rm sup}$. Hence $\mathscr P(t_L^-)=-t_L^-\chi_{\rm sup}$ and therefore
$t_L^-\leq t^-$ and $t_L^-=t^-$ holds. 
In the same way we obtain $t_L^+=t^+$. 
Then \eqref{press} also implies $\mathscr{P}(t)= \mathscr P_L(t)$ for every $t\in(t^-,t^+)$.
From the uniqueness, the equilibrium state for the potential $-t\log|Df|$ among liftable measures has to coincide with
$\mu_t$.

\begin{lemma} \label{unequal}
 For every $t\in (t_L^-,t_L^+)$, $\mathscr P_L(t)\neq\mathscr P_{\lnot L}(t)$.
 \end{lemma}
 \begin{proof}
 Since the supports of non-liftable measures 
are contained in a hyperbolic set, for any fixed $t\in\mathbb R$ the map
$\mu\in\mathcal M(f)\setminus\mathcal M_L(f)\mapsto h(\mu)-t\chi(\mu)$ is upper semi-continuous.
Hence, the supremum in $\mathscr P_{\lnot L}(t)$ is attained by a measure supported on the hyperbolic set.
On the other hand, the supremum in $\mathscr P_{\lnot L}(t)$ is attained by Proposition \ref{uniquel}.
If $\mathscr P_L(t)=\mathscr P_{\lnot L}(t)$ then there exist two distinct equilibrium states
for the potential $-t\log|Df|$, a contradiction to the uniqueness result \cite[Appendix A]{IomTod10}.
 \end{proof}
 In what follows we assume $\mathscr P_L(t_0)< \mathscr P_{\lnot L}(t_0)$ for some $t_0\in (t_L^-,t_L^+)$ and derive a contradiction.
 From Lemma \ref{unequal} and the continuity
of the pressures,
 $\mathscr P_L(t)<\mathscr P_{\lnot L}(t)$
 holds for every $t\in (t_L^-,t_L^+)$. In particular, 
 \begin{equation}\label{equ1}\mathscr P_{\lnot L}(0)=h_{\rm top}(f),\end{equation}
 where $h_{\rm top}(f)$ denotes the topological entropy of $f$.
We also have \begin{equation}\label{equ2}
\mathscr P_{\lnot L}(1)<0,\end{equation}
for otherwise Ruelle's inequality \cite{Rue78a} implies  $\mathscr P_{\lnot L}(1)=0$.
The measure which attains this supremum is an acip, and has to contain the critical point $c$ in its support.
This is a contradiction.

Let $\tilde f$ be an $S$-unimodal map such that the following holds:
(i) $D\tilde f(c)=0$ and $ D^2\tilde f(c)\neq0$;
(ii) $\tilde f(c)=f(c)$;
(iii) $\tilde f|_{X\setminus A}=f|_{X\setminus A}$.
Define $\tilde{\mathscr P}_L(t)$ and $\tilde{\mathscr P}_{\lnot L}(t)$
in the same way as $\mathscr{P}_L(t)$ and $\mathscr{P}_{\lnot L}(t)$,
replacing $f$ by $\tilde f$.
Since $\tilde f$ and $f$ have the same kneading sequence, 
they have the same topological entropy.
From this and \eqref{equ1},
\begin{equation}\label{equ3}\mathscr P_{\lnot L}(0)\geq\tilde{\mathscr P}_{L}(0).\end{equation}
Since $\tilde f$ satisfies the Collet-Eckmann condition \cite{ColEck83}, it has an acip.
It is liftable, and so 
\begin{equation}\label{equ4}\tilde{\mathscr P}_L(1)=0.\end{equation}
Now, set $\varphi(t)=\mathscr P_{\lnot L}(t)-\tilde{\mathscr P}_{L}(t)$. 
From \eqref{equ2}  \eqref{equ3} \eqref{equ4}, $\varphi(0)\geq0$ and
$\varphi(1)<0$.
Hence, there exists $t_1\in[0,1)$ such that
$\varphi(t_1)=0$.
Since $\mathscr P_{\lnot L}(t_1)= \tilde{\mathscr P}_{\lnot L}(t_1)$
it follows that $\tilde f$ has two distinct equilibrium states 
for $-t_1\log|D\tilde f|$, a contradiction to \cite[Theorem 7.7]{PesSen08}.
The analyticity of the pressure is a by-product.
Theorem A-4 follows from Lemma \ref{positive} and Lemma \ref{lack}.
This finishes the proof of Theorem A.
\end{proof}

\subsection{Decay of correlations for the acip}\label{bc}
Theorem B and Theorem C are direct consequences of the tail estimates in Proposition \ref{inducep} together with known results.

  \begin{proof}[Proof of Theorem B]
  The upper bounds on decay of correlation for $\mu_{\rm ac}$ now follow from 
  Proposition \ref{inducep}-2, Proposition \ref{inducep}-3 and the results in
  \cite{You99}.
\end{proof}

 \begin{proof}[Proof of Theorem C]
 From Proposition \ref{inducep}-3, for every $\beta\in(v(c),1)$
 with $-1/\beta<-1/v(c)$
  there exist constants $C\geq1$ such that for every $n\geq1$,
 \begin{equation*}\label{integration}
   C^{-1}n^{1-\frac{1}{v(c)}}\leq\sum_{k=n}^\infty |\{R>k\}|\leq  Cn^{1-\frac{1}{\beta}}.
  \end{equation*}
  The desired lower bound follows from
  \cite[Theorem 1.3]{Gou04}. 
 \end{proof}

  \subsection{Lyapunov exponents of limit measures}\label{lay}
For the proof of Theorem D  we need the next lemma which 
bounds the amount of drop of Lyapunov exponents
  of measures in the weak* limit.

\begin{lemma}\label{sentak1}
Let $f$ be a unimodal map of class $C^2$ which satisfies the Misiurewicz condition.
Let $x\in X$, and let $\{\mu_k\}_k$ be a sequence of ergodic measures
 in $\mathcal M(f)$ such that
 $\mu_k\to\mu\in\mathcal M(f)$ 
weakly as $k\to+\infty $, where
$\mu=u\nu+(1-u)\nu_\bot$, $\nu,\nu_\bot\in\mathcal M(f)$,
$\nu(\omega(c))=1$, $\nu_{\bot}(\omega(c))=0$ and $0\leq u\leq 1$.
Then $$\liminf_{k\to+\infty}\chi(\mu_k)\geq(1-u)\chi(\nu_{\bot}).$$
\end{lemma}

\begin{proof}
If there exist infinitely many $k$ such that 
the support of $\mu_k$ is contained in $\omega(c)$, then $u=1$
and the inequality holds.
In what follows we assume the number of such $k$ is finite.

For $x\in X$ and $r>0$ define $B_r(x)=[x-r,x+r]\cap X$.
For each integer $m\geq1$
fix $\alpha_m>0$ such that $B_{\alpha_m}(c)\cap\omega(c)=\emptyset$, $\alpha_m\to0$ as $m\to\infty$ and
$$\inf_{x\in B_{\alpha_m}(c)\setminus\{c\}}\left|D\ell(x)\log|x-c|+\frac{\ell(x)}{x-c}\right|\geq m^2.$$
This makes sense by virtue of (M1) and (F1).
Set $S=\{n\geq1\colon |Df(f^n(c))|<2\}.$
For each $n\in S$ define 
$k(n)=\min\{i>1\colon |Df^{i}(f^n(c))|\geq2\}.$
Set
$$V_m=B_{\alpha_m}(c)\bigcup\left(\bigcup_{n\in S}
\bigcup_{i=0}^{k(n)-1}f^{i}(B_{1/m}(f^{n}(c)))\right)\bigcup\left(\bigcup_{n\notin S} B_{1/m}(f^n(c))\right).$$
Since $\sup_{n\in S}k(n)<+\infty$ from (M1) we have $$\bigcap_{m\geq1}V_m=\overline{\{f^n(c)\colon n\geq0\}}.$$ 
From the bounded distortion, the following holds for 
  sufficiently large $m$: 
  for every $n\in S$ and every
$x\in B_{1/m}(f^n(c))$, $|Df^{k(n)}(x)|\geq1$.
For every $n\geq1$ such that $n\notin S$ and every 
$x\in B_{1/m}(f^n(c))$, $|Df(x)|\geq1$.

Note that $\{V_m\}_{m\geq1}$ has the following property: 
there exists $m_0\geq1$
such that  if $m\geq m_0$, $x\in X$ and $q\geq1$ are such that 
$x,f(x),\ldots,f^{q-1}(x)\in V_m$ and $f^{q}(x)\notin V_m$, then $|Df^{q}(x)|\geq 1.$
If $x\in B_{1/m}(f^n(c))$ holds for some $n\geq1$, then this 
follows from the definition of $V_m$.
If $x\in B_{\alpha_m}(c)$, then
 since $\bigcup_{n\geq1}B_{1/m}(f^n(c))\subset V_m,$
$|f^q(x)-f^q(c)|\geq1/m$ holds.
Hence \begin{align*}
|Df^{q}(x)|&=|Df^{q-1}(f(x))||Df(x)|\\
&\sim |Df^{q-1}(f(x))|     |f(x)-f(c)|\left|D\ell(x)\log|x-c|+\frac{\ell(x)}{x-c}\right|\\
&\sim |f^q(x)-f^q(c)|\left|D\ell(x)\log|x-c|+\frac{\ell(x)}{x-c}\right|.
\end{align*}
This number is comparable to $m$, and therefore
$|Df^q(x)|\geq1$ provided $m$ is sufficiently large.

For each $m\geq m_0$ such that $|Df|<1$ on $B_{\alpha_m} (c)$,
let $\rho_m$ be a continuous function on $X$ such that $\rho_m\geq\log|Df|$,
$\rho_m < 0$ on $B_{\alpha_m} (c)$ and
 $\rho_m=\log|Df|$ on $X\setminus B_{\alpha_m}(c)$. 
  Let $1_m$ denote the indicator function of $V_m$.
For each $\mu_k$ take a point $x_k\in X$ such that the following holds:
$$\lim_{N\to\infty}\frac{1}{N}\sum_{n=0}^{N-1}\log|Df(f^n(x_k))|= \chi(\mu_k);$$
$$\lim_{N\to\infty}\frac{1}{N}\sum_{n=0}^{N-1}1_m(f^n(x_k)\rho_m(f^n(x_k))= \int1_m\rho_md\mu_k;$$
$$\lim_{N\to\infty}\frac{1}{N}\sum_{n=0}^{N-1}\varphi(f^n(x_k))= \int\varphi d\mu_k
\enspace\text{for every continuous $\varphi\colon X\to\mathbb R$}.$$
Since the support of  $\mu_k$ is not contained in $\omega(c)$, 
for every sufficiently large $m$, $f^n(x_k)\notin V_m$ holds for infinitely many $n\geq0$.
Let $\{n(\ell)\}_{\ell\geq1}$ denote the subsequence obtained by aligning the elements of the set
$\{n\geq0\colon f^n(x_k)\notin V_m\}$ in the increasing order.
The property of $\{V_m\}_{m\geq1}$ implies
\begin{align*}\sum_{n=0}^{n(\ell)-1}\log|Df(f^n(x_k))|&\geq
\sum_{\stackrel{0\leq n\leq n(\ell)-1}{f^n(x_k)\notin V_m}}\log|Df(f^n(x_k))|\\
&=\sum_{\stackrel{0\leq n\leq n(\ell)-1}{f^n(x_k)\notin V_m}}\rho_m(f^n(x_k))\\
&=\sum_{n=0}^{n(\ell)-1}\rho_m(f^n(x_k))
-\sum_{n=0}^{n(\ell)-1}1_m(f^n(x))\rho_m(f^n(x_k)).
\end{align*}
On the second summand of the last line, 
\begin{align*}\lim_{\ell\to\infty}\frac{1}{n(\ell)}\sum_{n=0}^{n(\ell)-1}1_m(f^n(x_k))\rho_m(f^n(x_k))
&=\int 1_m\rho_md\mu_k\\
&=\int_{V_m}\rho_md\mu_k\\
&\leq\int_{V_m\setminus B_{\alpha_m}(c)}\rho_md\mu_k\\
&=\int_{V_m\setminus B_{\alpha_m}(c)}\log|Df|d\mu_k.
\end{align*}
The inequality holds provided $m$ is sufficiently large so that $\rho_m$ is negative on $B_{\alpha_m} (c)$.
Hence
\begin{align*}\chi(\mu_k)
&=\lim_{\ell\to\infty}\frac{1}{n(\ell)}\sum_{n=0}^{n(\ell)-1}\log|Df(f^n(x_k))|\\
&\geq \int\rho_md\mu_k-\int_{V_m\setminus B_{\alpha_m}(c)}\log|Df|d\mu_k.\end{align*}
Since $\mu_k\to\mu$ weakly as $k\to\infty$, 
$$\liminf_{k\to\infty}\chi(\mu_k)\geq \int\rho_md\mu-   u\int_{V_m\setminus B_{\alpha_m}(c)}\log|Df|d\nu-(1-u)\int_{V_m\setminus 
B_{\alpha_m}(c)}\log|Df|d\nu_\bot.$$
From the Dominated Convergence Theorem, $\int\rho_md\mu\to\chi(\mu)$
as $m\to\infty$. Since $\omega(c)$ is contained in $V_m\setminus B_{\alpha_m}(c)$,
the second integral is equal to $\chi(\nu)$.
From $\nu_\bot(\omega(c))=0$ and  $\nu_\bot(\overline{\{f^n(c)\colon n\geq1\}}\setminus\omega(c))=0$,
the third integral goes to $0$ as $m\to\infty$.
This finishes the proof of Lemma \ref{sentak1}.
\end{proof}

\subsection{Freezing temperature limit}\label{flimit}

We are in position to complete the proof of Theorem D.

\begin{proof}[Proof of Theorem D]
From the assumption on $f$, the acim is a finite measure. By Proposition \ref{lackexp}, $t^+=1$.
Let $\{\mu_{t_k}\}_{k\geq0}$ be a sequence of equilibrium states
such that $t_k\nearrow 1$ and $\mu_{t_k}$ converges weakly to a measure $\mu$ as $k\to\infty$.
Let  $\delta_0$ denote the Dirac measure at the fixed point $0$ and 
write $\mu=u\delta_0+(1-u)\nu$, $\nu\neq\delta_0$.
Since $\chi_{\rm inf}=0$ by Lemma \ref{lackexp}, $\mathscr{P}(1)=\displaystyle{\lim_{k\to\infty}\mathscr{P}(t_k)}=0$.
From the upper semi-continuity of entropy and Lemma \ref{sentak1},
\begin{align*}
0=\lim_{k\to\infty}\mathscr{P}(t_k)&\leq\limsup_{n\to\infty}h(\mu_{t_k})-
\liminf_{k\to\infty}\chi(\mu_{t_k})\\
&\leq h(\mu)-(1-u)\chi(\nu)\\
&=(1-u)(h(\nu)-\chi(\nu)).
\end{align*}
If $u=1$ then $\mu=\delta_0$.
If $u\neq1$ then $h(\nu)-\chi(\nu)\geq0$, and thus
$\nu=\mu_{\rm ac}$. Hence in all cases
 $\mu=u\delta_0+(1-u)\mu_{\rm ac}$. 

To conclude $u=0$, for each integer $n>2$ let $\varphi_n$ be a continuous function on $X$ such that $0\leq\varphi_n\leq1$, $\varphi_n=1$ on $[0,1/n]$
and $\varphi_n=0$ on $[2/n,1]$.
We have $\int\varphi_nd\delta_0=1$ for every $n$ and $\int\varphi_nd\mu_{\rm ac}\to0$ as $n\to\infty$.
Below we show
\begin{equation}
\label{final} \limsup_{n\to\infty} \lim_{k\to\infty}\int \varphi_nd\mu_{t_k}=0.
\end{equation}
 \eqref{final} gives  $\int\varphi_n d\mu_{t_k}\to0$ as $k\to\infty$.
 Since $\mu_{t_k}\to\mu$, $\int\varphi_n d\mu_{t_k}\to\int\varphi_n d\mu$.
If $u>0$, then $\int\varphi_n d\mu>u/2>0$ holds for sufficiently large $n$ and
we obtain a contradiction.

It is left to show \eqref{final}.
We first treat the case where 
 there exists a positive $C^3$ function $v$ on $X$ such that
$v(c)\in(0,1/2)$ and
 $\ell(x)=|x-c|^{-v(x)}$ for all $x$ near $c$.
Lastly we indicate necessary modifications to treat the other case.

Let $a_0^+$ be the fixed point of $f$ which is not $0$, and put
 $I=(a_0^-,a_0^+)$ (See Remark \ref{chebyshev}).
Let $(I,\mathcal W,R)$ be an inducing scheme for which the conclusions of Proposition \ref{inducep} hold.
For each $\mu_t$ let $\widehat\mu_{t}$ denote the Gibbs measure such that $\mathcal L(\widehat\mu_{t})=\mu_{t}$.
  Choose $\beta>v(c)$ and $t_0\in(0,1)$ such that $\frac{t_0}{\beta}>2$.
 For every $J\in\mathcal W$ with $R(J)=n$ we have
$J\subset [a_{n-1}^-,a_{n-1}^+]$,  see Remark \ref{chebyshev}.
From \eqref{in} there exists $C>0$ such that 
$a_{n-1}^+-a_{n-1}^-\leq Cn^{-\frac{1}{\beta}}$
holds for every $n\geq1$.
 Then
  \begin{align*}
  \widehat\mu_{t}(J)& \leq   C_2(t)\sup_{x\in J} |D\widehat f(x)|^{-t}\\
  &\leq C_2(t)K_\tau^t|J|^t\\
  &\leq  C_2(t)(CK_\tau)^{t}n^{-\frac{t}{\beta}}.
\end{align*}
Here, $C_2(t)$ denotes the constant in \eqref{gibbs}
with $\widehat\mu_*=\widehat\mu_t$.
The first inequality is because $\mathscr{P}(t,\mathscr{P}_L(t))=0$ and $\mathscr{P}_L(t)>0$.



Let $b_n$ be a sequence of integers such that $b_n\to\infty$ and
$[0,2/n]\cap\bigcup_{k=0}^{R(J)-1}f^k(J)=\emptyset$ for every $J\in\mathcal W$ such that $R(J)\leq b_n$.
We have $\#S(n)=2$ for every $n\geq2$, see Remark \ref{chebyshev}.
  For every large $k$ with $t_k\in(t_0,1)$ we have
\begin{align*}
\int\varphi_n d\mu_{t_k}&=\frac{1}{\int Rd\widehat\mu_{t_k}}\sum_{R(J)>b_n} R(J)\widehat\mu_{t_k}(J)\\
&\leq\sum_{R(J)>b_n} R(J)\widehat\mu_{t_k}(J)\\
&=\sum_{n>b_n}\sum_{\stackrel{J\in\mathcal W}{R(J)=n}} n\widehat\mu_{t_k}(J)\\
&\leq C_2(t)(CK_\tau)^{t}\sum_{n>b_n}\#S(n)n^{1-\frac{t_k}{\beta}}\\
&\leq 2C_2(t)(CK_\tau)^{t}\sum_{n>b_n}n^{1-\frac{t_0}{\beta}}.
\end{align*}
A close inspection of \cite[Proposition 2, Theorem 8]{Sar99} shows that $\sup_{t\in (t_0,1)}C_2(t)<+\infty$.
The series in the last line is finite and so goes to $0$ as $n\to+\infty$.
Hence \eqref{final} holds.
To show \eqref{final} for the case $\ell(x)\sim|\log|x-c||^\alpha$ for some $\alpha>0$,
instead of \eqref{in} we use the better estimate \eqref{in'} and proceed in the same way.
  \end{proof}


\end{document}